\documentclass{article}

\usepackage{amssymb, amsmath, amsthm}
\usepackage{MnSymbol}
\usepackage{esint}
\usepackage{wasysym}
\usepackage{graphicx}
\usepackage{color}
\usepackage{cite}

\theoremstyle{plain}
\newtheorem{thm}{Theorem}

\newtheorem{cor}[thm]{Corollary}

\newtheorem{prop}[thm]{Proposition}

\theoremstyle{definition}
\newtheorem{defn}[thm]{Definition}
\theoremstyle{remark}
\newtheorem{rem}[thm]{Remark}

\newcommand{\norm}[1]{\left\Vert#1\right\Vert}

\newcommand{\Real}{\mathbb R}

\begin{document}
\bibliographystyle{plain}
\title{Non-uniqueness for the Euler equations: the effect of the boundary}
\author{Claude Bardos\footnote{Laboratoire J.-L. Lions, Universit\'{e} de Paris VII ``Denis Diderot'', Paris, France} \and  L\'aszl\'o Sz\'ekelyhidi Jr.\footnote{Mathematisches Institut, Universit\"at Leipzig, Germany}  \and Emil Wiedemann\footnote{Department of Mathematics, University of British Columbia, and Pacific Institute for the Mathematical Sciences, Vancouver, B.C., Canada}}
\date{}
\maketitle
\begin{abstract}

We consider rotational initial data for the two-dimensional incompressible Euler equations on an annulus. Using the convex integration framework, we show that there exist infinitely many admissible weak solutions (i.e. such with non-increasing energy) for such initial data. As a consequence, on bounded domains there exist admissible weak solutions which are not dissipative in the sense of P.-L. Lions, as opposed to the case without physical boundaries. Moreover we show that admissible solutions are dissipative provided they are H\"{o}lder continuous near the boundary of the domain.    
\end{abstract}

\begin{center}
\it{Dedicated to the memory of Professor Mark Iosifovich Vishik}
\end{center}

\section{Introduction}
The study of weak solutions of the incompressible Euler equations is motivated by (at least) two aspects of fluid flow: the presence of instabilities, most notably the Kelvin-Helmholtz instability, and fully developed 3-dimensional turbulence. Concerning the latter, an important problem arises in connection with the famous $5/3$ law of Obukhov-Kolmogorov and the conjecture of Onsager regarding energy conservation. We refer to \cite{EyinkSreenivasan,bardostiti2} and \cite{bdlsz,daneri,isett} for more information and recent progress regarding this problem. 

Concerning the former, it has been the subject of intensive research to define a physically meaningful notion of weak solution, that can capture the basic features of such instabilities and be analytically well behaved at the same time. Due to the lack of an analogous theorem to the existence of Leray-Hopf weak solutions of the Navier-Stokes equations, several weaker notions have been considered.

Dissipative solutions of the incompressible Euler equations were introduced by P.-L. Lions \cite{lions} as a concept of solution with two desirable properties: (i) existence for arbitrary initial data, and (ii) weak-strong uniqueness, meaning that a dissipative weak solution agrees with the strong solution as long as the latter exists. Dissipative solutions have been shown to arise, among others, as viscosity \cite{lions} or hydrodynamic \cite{saintraymond} limits of the incompressible Euler equations.
The major draw-back of dissipative solutions is that, in general, the velocity field does not solve the Euler equations in the sense of distributions. 

Weak solutions (i.e. distributional solutions with some additional properties) on the other hand have been constructed by various techniques, see\cite{scheffer,shnirel1,shnirel2,euler1,euler2,euleryoung,eulerexistence,szlecturenotes}. Many of these results come with a high level of non-uniqueness, even violating the weak-strong uniqueness property - we refer to the survey \cite{hprinciple}. In particular, in \cite{eulerexistence} the existence of global in time weak solutions was shown for arbitrary initial data.

Due to the high level of non-uniqueness, a natural question is whether there are any selection criteria among weak solutions. With this regard, it has been noted in \cite{DuchonRobert,euler2} that, in the absense of boundaries a weak solution is dissipative in the sense of Lions, provided the weak energy inequality 
\begin{equation}\label{ei}
\int |v(x,t)|^2\,dx\leq \int |v(x,0)|^2\,dx\qquad\textrm{ for almost every $t>0$}
\end{equation}    
holds. In \cite{euler2} this condition is referred to as an admissibility condition, in analogy with the entropy condition used in hyperbolic conservation laws \cite{DafermosBook}. Admissibility turned out to be a useful selection criterion among weak solutions, since already in the weak form in \eqref{ei} it implies the weak-strong uniqueness property of dissipative solutions (stronger versions of the energy inequality are discussed in \cite{euler2}). This is even the case not just for distributional solutions but also for measure-valued solutions, see \cite{brenierdelellissz}.  

Despite the weak-strong uniqueness property, there exists a large, in fact $L^2$ dense set of initial data on the whole space or with periodic boundary conditions \cite{euleryoung} (see also \cite{szlecturenotes}), for which the initial value problem admits infinitely many {\it admissible} weak solutions. Such initial data, called ``wild initial data'', necessarily has to be irregular. 

The non-uniqueness of admissible weak solutions is intimately related to the presence of instabilities. For instance, in \cite{vortexpaper} the non-uniqueness of admissible weak solutions was shown for the flat vortex sheet initial data
\begin{equation}\label{e:flat}
v_0(x)=\begin{cases}
e_1 & \text{if $x_d\in(0,\frac{1}{2})$}\\
-e_1 & \text{if $x_d\in(-\frac{1}{2},0)$,}
\end{cases}
\end{equation} 
extended periodically to the torus $\mathbb{T}^d$. Note that the stationary vector field is an obvious solution in this case, but the statement in \cite{vortexpaper} is that there exist infinitely many non-stationary solutions. A common feature in these solutions is that for time $t>0$ they exhibit an expanding "turbulent" region around the initial vortex sheet, much akin to the propagation of singularity in the classical Kelvin-Helmholtz problem. Further examples of this nature appeared in \cite{shearflow} and recently in \cite{isentropic} for the compressible Euler system.

Motivated by the idea that it is the underlying Kelvin-Helmholtz instability that is responsible for the non-uniqueness of admissible weak solutions, we study in this note the case of domains with boundary. We show that the presence of a (smooth) boundary can lead to the same effect of an expanding turbulent region as in \cite{vortexpaper}. As a corollary, we observe that admissibility does not imply the weak-strong uniqueness property in domains with boundary. 

\section{Statement of the main results}

\subsection{Formulation of the equations}

We study weak solutions of the initial and boundary value problem for the incompressible Euler equations
\begin{equation}\label{euler}
\begin{aligned}
\partial_tv+v\cdot\nabla v+\nabla p&=0\\
\operatorname{div}v&=0\\
v|_{t=0}&=v_0\\
\end{aligned}
\end{equation}
complemented with the usual kinematic boundary condition
$$
v|_{\partial\Omega}\cdot\nu=0.
$$ 
Here, $\Omega\subset\Real^d$, $d\geq 2$, is a domain with sufficiently smooth boundary, $T>0$ a finite time, $v:\Omega\times[0,T)\rightarrow\Real^d$ the velocity field, $p:\Omega\times(0,T)\rightarrow\Real$ the scalar pressure, $v_0$ the initial velocity and $\nu$ the inner unit normal to the boundary of $\Omega$.

In order to give the precise definition of weak solutions, consider the space of solenoidal vectorfields on $\Omega$ (cf. Chapter III of \cite{galdibook}),  
\begin{equation*}
\begin{aligned}
H(\Omega)=\big\{v\in L^2&(\Omega;\Real^d):\int_{\Omega}v\cdot\nabla p dx=0\\
  &\text{for every $p\in W^{1,2}_{loc}(\Omega)$ such that $\nabla p\in L^2(\Omega)$}\big\}.
\end{aligned}
\end{equation*}

Let $v_0\in H(\Omega)$. An {\it admissible weak solution} of (\ref{euler}) with initial data $v_0$ is defined to be a vectorfield $v\in L^{\infty}(0,T;H(\Omega))$ such that for every test function $\phi\in C_c^{\infty}(\Omega\times[0,T);\Real^2)$ with $\operatorname{div}\phi=0$, we have
\begin{equation*}
\int_0^T\int_{\Omega}\left(\partial_t\phi\cdot v+\nabla\phi:v\otimes v\right)dxdt+\int_{\Omega}v_0(x)\cdot\phi(x,0)dx=0,
\end{equation*} 
and the energy inequality \eqref{ei} holds.

We remark in passing that in fact one may assume that admissible weak solutions are in the space $C([0,T);H_w(\Omega))$, where $H_w(\Omega)$ is the space $H(\Omega)$ equipped with the weak $L^2$-topology. Indeed, dissipative solutions of Lions are also defined in this space. Nevertheless, for simplicity we will just treat the velocity fields as elements in the larger space $L^{\infty}(0,T;H(\Omega))$.

\subsection{Rotationally symmetric data}

In the present paper, we consider rotationally symmetric initial data in two dimensions. It should be noted that the restriction to 2 dimensions is purely for simplicity of presentation - the constructions and the methods can be easily extended to higher dimensions. Similarly, we will consider as domain an annulus purely for simplicity of presentation - the nontrivial topology of the domain does not play a role in our results. 

By ``rotational'' we mean initial data of the form 
\begin{equation}\label{rotational}
v_0(x)=\alpha_0(r)(\sin\theta,-\cos\theta)
\end{equation} 
on an annulus
\begin{equation}\label{annulus}
\Omega=\{x\in\Real^2: \rho<|x|<R\},
\end{equation}
where $0<\rho<R<\infty$. Vector fields as in (\ref{rotational}) are known to define stationary solutions to the Euler equations regardless of the choice of $\alpha_0$, and are frequently used as explicit examples in the study of incompressible flows \cite{amick, shnirel3, bertozzimajda}.

Fix a radius $r_0$ with $\rho<r_0<R$ and consider the initial data on the annulus given by \eqref{rotational} with
\begin{equation}\label{initial}
\alpha_0(r)=\begin{cases}-\frac{1}{r^2} &\text{if $\rho<r<r_0$}\\
\frac{1}{r^2} &\text{if $r_0<r<R$},
\end{cases}
\end{equation}
which corresponds to a rotational flow with a jump discontinuity on the circle $\{r=r_0\}$. 

\begin{thm}\label{wildrotation}
Let $\Omega$ be an annulus as in (\ref{annulus}), $T>0$ a finite time, and $v_0$ be rotational as in (\ref{rotational}) and (\ref{initial}). Apart from the stationary solution $v(\cdot,t)=v_0$, there exist infinitely many non-stationary admissible weak solutions of the Euler equations on $\Omega\times(0,T)$ with initial data $v_0$. Among these, infinitely many have strictly decreasing energy, and infinitely many conserve the energy.
\end{thm} 
Our proof, given in Section \ref{nonuniqueness} below, relies on the techniques from \cite{euler2} and is similar to the construction in \cite{vortexpaper}. 

Regarding the quest for suitable selection principles, a much-discussed criterion is the viscosity solution, defined to be a solution obtained as a weak limit of Leray-Hopf solutions as viscosity converges to zero.  In the case of the initial data in \eqref{e:flat} it is an easy exercise (see for instance \cite{shearflow}) to show that the viscosity solution agrees with the stationary solution. In the rotational case \eqref{initial} the same is true, as we show in Section \ref{viscosity} below:
\begin{prop}\label{introuniqueness}
Let $\Omega\subset\Real^2$ be an annulus and let initial data be given by (\ref{rotational}). Then every sequence of Leray-Hopf solutions of the Navier-Stokes equations with viscosities tending to zero which correspond to this initial data will converge strongly to the stationary solution $v(\cdot,t)=v_0$ of the Euler equations. 
\end{prop}

Finally, we discuss the relation between admissible weak solutions and dissipative solutions of Lions in bounded domains. For the convenience of the reader we recall in Section \ref{dissipative} the precise definition of dissipative solutions. As a corollary to Theorem \ref{wildrotation} we show in Section \ref{dissipative} that, contrary to the case without boundaries, admissible weak solutions need not be dissipative:
\begin{cor}\label{cor}
On $\Omega$ there exist admissible weak solutions which are not dissipative solutions.   
\end{cor}
Corollary \ref{cor} says that in the presence of boundary the weak-strong uniqueness might fail for admissible weak solutions. On the technical level the explanation for this lies in the observation that the notion of strong solution in a bounded domain does not allow any control of the boundary behaviour. Therefore in Section \ref{holder} we study what happens when additional boundary control is available: \begin{thm}\label{criterion}
Let $\Omega\subset\Real^2$ be a bounded domain with $C^2$ boundary. Suppose $v$ is an admissible weak solution of (\ref{euler}) on $\Omega$ for which there exists some $\delta>0$ and $\alpha>0$ such that $v$ is H\"{o}lder continuous with exponent $\alpha$ on the set
\begin{equation*}
\Gamma_{\delta}=\left\{x\in\overline{\Omega}:\operatorname{dist}(x,\partial\Omega)<\delta\right\},
\end{equation*} 
uniformly in $t$. Then $v$ is a dissipative solution. 
\end{thm}

\section{Subsolutions and convex integration}\label{s:subsolution}

In order to prove Theorem \ref{wildrotation} we recall the basic framework developed in \cite{euler1,euler2}, with slight modifications to accomodate for domains with boundary. 
For further details we refer to the survey \cite{hprinciple} and the recent lecture notes \cite{szlecturenotes}.

To start with, recall the definition of subsolution. To this end let us fix a non-negative function 
$$
\overline{e}\in L^\infty(0,T;L^1(\Omega)),
$$
which will play the role of the (kinetic) energy density. We will work in the space-time domain 
$$
\Omega_T:=\Omega\times (0,T),
$$
where $\Omega\subset\Real^d$ is either an open domain with Lipschitz boundary or $\Omega=\mathbb{T}^d$.

\begin{defn}[Subsolution]\label{d:subsolution}
A subsolution to the incompressible Euler equations with respect to the kinetic energy density $\overline{e}$
is a triple 
$$
(\bar{v},\bar{u},\bar{q}):\Omega_T\to \Real^d\times\mathcal{S}^{d\times d}_0\times\Real
$$
with $\bar{v}\in L^\infty(0,T;H(\Omega)),\,\bar{u}\in L^1_{loc}(\Omega_T),\, \bar{q}\in \mathcal{D}'(\Omega_T)$, such that
\begin{equation}\label{e:LR}
\left\{\begin{array}{l}
\partial_t \bar{v}+\mathrm{div }\bar{u}+\nabla \bar{q} =0\\
\mathrm{div }\bar{v} =0,
\end{array}\right. \qquad \mbox{in the sense of distributions;}
\end{equation}
and moreover
\begin{equation}\label{e:CR}
\bar{v}\otimes \bar{v}-\bar{u}\leq \tfrac{2}{d}\overline{e}\,I\quad\textrm{ a.e. $(x,t)$.}
\end{equation}
\end{defn}
Here $\mathcal{S}^{d\times d}_0$ denotes the set of symmetric traceless $d\times d$ matrices and $I$ is the identity matrix.
Observe that subsolutions automatically satisfy $\tfrac{1}{2}|\bar{v}|^2\leq \overline{e}$ a.e. If in addition \eqref{e:CR}
is an equality a.e. then $\bar{v}$ is a weak solution of the Euler equations.

A convenient way to express the inequality \eqref{e:CR} is obtained by introducing the {\it generalized energy density} 
\begin{equation*}
e(\bar{v},\bar{u})=\frac{d}{2}|\bar{v}\otimes \bar{v}-\bar{u}|_{\infty},
\end{equation*}
where $|\cdot|_{\infty}$ is the operator norm of the matrix ($=$ the largest eigenvalue for symmetric matrices). The inequality \eqref{e:CR} can then be equivalently written as 
\begin{equation}\label{e:CREuler1}
e(\bar{v},\bar{u})\leq \bar{e}\textrm{ a.e. }
\end{equation}

The key point of convex integration is that a {\em strict inequality} instead of \eqref{e:CR} gives enough room so that high-frequency oscillations can be ``added'' on top of the subsolution -- of course in a highly non-unique way -- so that one obtains weak solutions. It is important also to note that, since in the process of convex integration only compactly supported (in space-time) perturbations are added to the subsolution, the boundary and initial conditions of the weak solutions so obtained agree with the corresponding data of the subsolution. This is the content of the following theorem, which is essentially Proposition 2 from \cite{euler2}. 

\begin{thm}[Subsolution criterion]\label{t:criterion}
Let $\overline{e}\in L^{\infty}(\Omega_T)$ and $(\overline{v},\overline{u},\overline{q})$ 
be a subsolution. Furthermore, let $\mathcal{U}\subset\Omega_T$ a subdomain such that
$(\overline{v},\overline{u},\overline{q})$ and $\overline{e}$ are continuous on $\mathcal{U}$ and
\begin{equation}\label{e:strict}
\begin{split}
e(\overline{v},\overline{u})&<\overline{e}\qquad \textrm{ on }\mathcal{U}\\
e(\overline{v},\overline{u})&=\overline{e}\qquad \textrm{ a.e. }\Omega_T\setminus \mathcal{U}
\end{split}
\end{equation}
Then there exist infinitely many weak solutions $v\in L^{\infty}(0,T;H(\Omega))$
of the Euler equations such that
\begin{align*}
v&=\overline{v} \qquad \textrm{ a.e. $\Omega_T\setminus \mathcal{U}$,}\\
\tfrac{1}{2}|v|^2&=\overline{e} \qquad\textrm{ a.e. $\Omega_T$,}\\
p&=\overline{q}-\tfrac{2}{d}\overline{e} \qquad\textrm{ a.e. $\Omega_T$}.
\end{align*}
If in addition
\begin{equation}\label{e:initialdatum}
\overline{v}(\cdot,t)\rightharpoonup v_0(\cdot)\textrm{ in }L^2(\Omega)\textrm{ as }t\to 0 ,
\end{equation}
then $v$ solves the Cauchy problem \eqref{euler}.
\end{thm}

We also refer to \cite{szlecturenotes}, where a detailed discussion of the convex integration technique can be found - in particular the above theorem is Theorem 7 of \cite{szlecturenotes}.

\section{Non-Uniqueness for Rotational Initial Data}\label{nonuniqueness}

In this section we wish to apply the framework of Section \ref{s:subsolution} to prove Theorem \ref{wildrotation}. Thus, we set
$$
\Omega:=\{x\in\Real^2:\,\rho<|x|<R\}
$$
to be an annulus, fix $r_0\in (\rho,R)$ and set
\begin{equation}\label{e:v0}
v_0(x)=\begin{cases}-\frac{1}{|x|^3}x^\perp&|x|<r_0,\\ \frac{1}{|x|^3}x^{\perp}&|x|>r_0,\end{cases}
\end{equation}
where $x^\perp=\begin{pmatrix}x_2\\-x_1\end{pmatrix}$. We will construct subsolutions by a similar method as in \cite{vortexpaper}.

Owing to Theorem \ref{t:criterion} of the previous section, it suffices to show the existence of 
certain subsolutions. We fix two small constants $\lambda>0$ ("turbulent propagation speed") and $\epsilon\geq 0$ ("energy dissipation rate"), to be determined later.

We look for subsolutions $(\bar{v},\bar{u},\bar{q})$ (c.f. Definition \ref{d:subsolution} - the energy density function $\bar{e}$ is still to be fixed) of the form
\begin{equation*}
\bar{v}(x,t)=\alpha(r,t)\begin{pmatrix}\sin\theta\\-\cos\theta\end{pmatrix},
\end{equation*}
where $\alpha(r,0)=\alpha_0(r)$ and $(r,\theta)$ denotes polar coordinates on $\Real^2$,
\begin{equation}\label{defu}
\begin{aligned}
\bar{u}(x,t)&=\left(\begin{array}{cc}\cos\theta & \sin\theta\\ 
\sin\theta & -\cos\theta\end{array}\right)\left(\begin{array}{cc}\beta(r,t) & \gamma(r,t)\\ 
\gamma(r,t) & -\beta(r,t)\end{array}\right)\left(\begin{array}{cc}\cos\theta & \sin\theta\\ 
\sin\theta & -\cos\theta\end{array}\right)\\
&=\left(\begin{array}{cc}\beta\cos(2\theta)+\gamma\sin(2\theta) & \beta\sin(2\theta)-\gamma\cos(2\theta)\\ 
\beta\sin(2\theta)-\gamma\cos(2\theta) & -\beta\cos(2\theta)-\gamma\sin(2\theta)\end{array}\right),
\end{aligned}
\end{equation}
and
\begin{equation*}
\bar{q}=\bar{q}(r).
\end{equation*}

As a side remark, note that the choice $\alpha(r,t)=\alpha_0(r)$ for all $t\geq0$, $\beta=-\frac{1}{2}\alpha^2$, $\gamma=0$, and 
\begin{equation}\label{pressure}
\bar{q}(r)=\frac{1}{2}\alpha^2+\int_{\rho}^r\frac{\alpha(s)^2}{s}ds
\end{equation}
yields the well-known stationary solution (the integral in the formula for $\bar{q}$ represents the physical pressure).

We insert this ansatz into \eqref{e:LR} to arrive at two equations. More precisely, using the formulas $\nabla_xr=\begin{pmatrix}\cos\theta\\ \sin\theta\end{pmatrix}$ and $\nabla_x\theta=\frac{1}{r}\begin{pmatrix}-\sin\theta\\ \cos\theta\end{pmatrix}$, we obtain
\begin{equation*}
\begin{aligned}
\partial_t\alpha\sin\theta &+\partial_r\beta\left[\cos\theta\cos(2\theta)+\sin\theta\sin(2\theta)\right]+\partial_r\gamma\left[\cos\theta\sin(2\theta)-\sin\theta\cos(2\theta)\right]\\
&+\frac{2}{r}\beta\left[\sin\theta\sin(2\theta)+\cos\theta\cos(2\theta)\right]+\frac{2}{r}\gamma\left[-\sin\theta\cos(2\theta)+\cos\theta\sin(2\theta)\right]\\
&+\partial_r\bar{q}\cos\theta=0
\end{aligned}
\end{equation*}
and
\begin{equation*}
\begin{aligned}
-\partial_t\alpha\cos\theta &+\partial_r\beta\left[\cos\theta\sin(2\theta)-\sin\theta\cos(2\theta)\right]+\partial_r\gamma\left[-\cos\theta\cos(2\theta)-\sin\theta\sin(2\theta)\right]\\
&+\frac{2}{r}\beta\left[-\sin\theta\cos(2\theta)+\cos\theta\sin(2\theta)\right]+\frac{2}{r}\gamma\left[-\sin\theta\sin(2\theta)-\cos\theta\cos(2\theta)\right]\\
&+\partial_r\bar{q}\sin\theta=0.
\end{aligned}
\end{equation*}
If we multiply the first equation by $\sin\theta$ and add it to the second one multiplied by $\cos\theta$, use the identities $\cos^2\theta-\sin^2\theta=\cos(2\theta)$ and $2\sin\theta\cos\theta=\sin(2\theta)$, and then separate by terms involving $\sin(2\theta)$ and $\cos(2\theta)$, respectively, we will eventually get the two equations
\begin{equation}\label{preburgers}
\begin{aligned}
\partial_r\beta+\frac{2}{r}\beta+\partial_r\bar{q}&=0\\
\partial_t\alpha+\partial_r\gamma+\frac{2}{r}\gamma&=0.
\end{aligned}
\end{equation}
It can be easily verified that these equations are equivalent to the original system \eqref{e:LR} for our ansatz.

If we set $\bar{q}(r)$ as in (\ref{pressure}) and $\beta=-\frac{1}{2}\alpha^2$, the first equation will be satisfied, in nice analogy with \cite{vortexpaper} (up to a sign). Also, the second equation is similar to \cite{vortexpaper}, but it involves the additional ``centrifugal'' term $\frac{2}{r}\gamma$. Therefore, we cannot simply set $\gamma=\frac{1}{2}\alpha^2$ as in \cite{vortexpaper} to obtain Burgers' equation.
However, observing that $\partial_r(r^2\gamma)=2r\gamma+r^2\partial_r\gamma$, we set 
$$
\alpha(r,t)=\frac{1}{r^2}f(r,t)
$$ 
and 
\begin{equation}\label{gamma}
\gamma=-\frac{\lambda}{2r^2}(1-f^2)=-\frac{\lambda}{2}\left(\frac{1}{r^2}-r^2\alpha^2\right),
\end{equation} 
so that the second equation in (\ref{preburgers}), after multiplication by $r^2$, turns into Burgers' equation
\begin{equation}\label{burgers}
\partial_tf+\frac{\lambda}{2}\partial_r(f^2)=0.
\end{equation}
The initial data (\ref{initial}) for $\alpha$ then
corresponds to
\begin{equation*}
f(r,0)=\begin{cases}-1 &\text{if $\rho<r<r_0$}\\
1 &\text{if $r_0<r<R$}.
\end{cases}
\end{equation*}
Then, for this data, Burgers' equation (\ref{burgers}) has a rarefaction wave solution for $t\in[0,T]$, provided $\lambda>0$ is sufficiently small (depending on $T$ and $\rho<r_0<R$), which can be explicitly written as
\begin{equation}\label{entropy}
f(r,t)=\begin{cases}-1 &\text{if $\rho<r<r_0-\lambda t$}\\
\frac{r-r_0}{\lambda t} &\text{if $r_0-\lambda t<r<r_0+\lambda t$}\\
1 &\text{if $r_0+\lambda t<r<R$.}
\end{cases}
\end{equation}

Therefore, by setting $\alpha(r,t)=\frac{1}{r^2}f(r,t)$ for $f$ as in (\ref{entropy}), $\beta=-\frac{1}{2}\alpha^2$, $\gamma$ as in (\ref{gamma}), and $\bar{q}$ as in (\ref{pressure}), we obtain a solution of the equations \eqref{e:LR} with initial data corresponding to \eqref{e:v0}. 

It remains to study the generalized energy. Since $\bar{u}$ is given by (\ref{defu}) and moreover
\begin{equation*}
\begin{aligned}
\bar{v}\otimes \bar{v}&=\alpha(r,t)^2\left(\begin{array}{cc}\cos^2\theta & -\sin\theta\cos\theta\\ 
-\sin\theta\cos\theta & \cos^2\theta\end{array}\right)\\
&=\left(\begin{array}{cc}\cos\theta & \sin\theta\\ 
\sin\theta & -\cos\theta\end{array}\right)\left(\begin{array}{cc}0 & 0\\ 
0 & \alpha(r,t)^2\end{array}\right)\left(\begin{array}{cc}\cos\theta & \sin\theta\\ 
\sin\theta & -\cos\theta\end{array}\right),
\end{aligned}
\end{equation*}  
and since the eigenvalues of a matrix are invariant under conjugation by an orthogonal transformation, in order to determine $e(\bar{v},\bar{u})=|\bar{v}\otimes \bar{v}-\bar{u}|_{\infty}$ it suffices to find the largest eigenvalue of
\begin{equation*}
\left(\begin{array}{cc}-\beta & -\gamma\\
-\gamma & \alpha^2+\beta
\end{array}\right)=\left(\begin{array}{cc}\frac{1}{2}\alpha^2 & \frac{\lambda}{2}\left(\frac{1}{r^2}-r^2\alpha^2\right)\\
\frac{\lambda}{2}\left(\frac{1}{r^2}-r^2\alpha^2\right) & \frac{1}{2}\alpha^2
\end{array}\right).
\end{equation*}
It is easily calculated, taking into account $|\alpha|\leq\frac{1}{r^2}$ and $\lambda\geq0$, that
\begin{equation}\label{genenergy}
\begin{aligned}
e(\bar{v},\bar{u})&=\frac{1}{2}\alpha^2+\frac{\lambda}{2}\left(\frac{1}{r^2}-r^2\alpha^2\right)\\
&=\frac{1}{2r^4}\left[1-(1-r^2\lambda)\left(1-f(r,t)^2\right)\right].
\end{aligned}
\end{equation}
Finally, we set 
\begin{equation*}
\bar{e}(r,t)=\frac{1}{2r^4}\left[1-\epsilon(1-r^2\lambda)\left(1-f(r,t)^2\right)\right]\,,
\end{equation*}
where $\epsilon$ is sufficiently small so that $\bar{e}>0$.
Observe that 
$$
e(\bar{v},\bar{u})\leq\bar{e}\leq\frac{1}{2}|v_0|^2\qquad\textrm{ in }\Omega_T.
$$
More precisely, we have the following result, summarizing the calculations in this section:

\begin{prop}\label{p:subsol}
For any choice of constants $\epsilon,\lambda$ satisfying 
\begin{align*}
0&<\lambda<\min\left\{\frac{1}{R^2},\frac{r_0-\rho}{T},\frac{R-r_0}{T}\right\},\\
0&\leq \epsilon< \frac{1}{1-\rho^2\lambda}\,
\end{align*}
there exists a subsolution $(\bar{v},\bar{u},\bar{q})$ in $\Omega_T$ with respect to the kinetic energy density 
\begin{equation*}
\bar{e}(r,t)=\frac{1}{2r^4}\left[1-\epsilon(1-r^2\lambda)\left(1-f(r,t)^2\right)\right]
\end{equation*}
and with initial data $\bar{v}(x,0)=v_0(x)$ from \eqref{e:v0}, such that, with
\begin{equation*}
\mathcal{U}:=\left\{x\in\Real^2:\,r_0-\lambda t<|x|<r_0+\lambda t\right\}
\end{equation*}
we have
\begin{align*}
&e(\bar{v},\bar{u})<\bar{e}\qquad\textrm{ in }\mathcal{U},\\
&e(\bar{v},\bar{u})=\bar{e}\qquad\textrm{ in }\Omega_T\setminus\mathcal{U}.
\end{align*}
\end{prop}

We can now conclude with the proof of Theorem \ref{wildrotation}. 

\begin{proof}[Proof of Theorem \ref{wildrotation}]
We apply Proposition \ref{p:subsol} above with $\epsilon\geq 0$ to obtain a subsolution $(\bar{v},\bar{u},\bar{q})$. According to Theorem \ref{t:criterion} with this subsolution, there exist infinitely many weak solutions $v\in L^{\infty}(0,T;H(\Omega))$ such that $|v|^2=2\bar{e}$ almost everywhere in $\Omega_T$ and with initial data $v_0$. To check that these are admissible, observe that
$$
\int_{\Omega}|v(x,t)|^2\,dx=\int_{\Omega}2\bar{e}(x,t)\,dx\leq \frac{1}{|x|^4}\,dx=\int_{\Omega}|v_0(x)|^2\,dx.
$$
Finally, observe that we obtain strictly energy-decreasing solutions by choosing $\epsilon>0$ and energy-conserving solutions for $\epsilon=0$.
\end{proof}

\section{Uniqueness of the Viscosity Limit}\label{viscosity}

\begin{proof}[Proof of Proposition \ref{introuniqueness}]
Consider the Navier-Stokes equations with viscosity $\epsilon>0$:
\begin{equation}\label{navierstokes}
\begin{aligned}
\partial_tv_{\epsilon}+v_{\epsilon}\cdot\nabla v_{\epsilon}+\nabla p_{\epsilon}&=\epsilon\Delta v_{\epsilon}\\
\operatorname{div}v_{\epsilon}&=0\\
v_{\epsilon}(\cdot,0)&=v_0\\
v_{\epsilon}|_{\partial\Omega}&=0.
\end{aligned}
\end{equation}
It is known that the Navier-Stokes equations in two space dimensions admit a unique weak solution (the Leray-Hopf solution) which satisfies the energy equality
\begin{equation*}
\frac{1}{2}\int_{\Omega}|v_{\epsilon}(x,t)|^2dx+\epsilon\int_0^t\int_{\Omega}|\nabla v_{\epsilon}(x,s)|^2dxds=\frac{1}{2}\int_{\Omega}|v_0(x)|^2dx
\end{equation*}
for every $t\in[0,T]$, see e.g. \cite{galdi} for details. It turns out that if the initial data $v_0$ has the rotational symmetry in \eqref{rotational}, then the (unique) Leray-Hopf solution will have the same symmetry. 

To show this, we take the ansatz 
\begin{equation}\label{ansatz}
v_{\epsilon}(x,t)=\alpha_{\epsilon}(r,t)\begin{pmatrix}\sin\theta\\-\cos\theta\end{pmatrix}
\end{equation}
and $p_{\epsilon}=p_{\epsilon}(r)$, again using polar coordinates. Insertion of this ansatz into the first equation of (\ref{navierstokes}) yields
\begin{equation*}
\begin{aligned}
\partial_t\alpha_{\epsilon}\sin\theta&-\frac{\alpha_{\epsilon}^2}{r}\cos\theta+\partial_rp_{\epsilon}\cos\theta\\
&=\epsilon\left(\frac{\partial_r\alpha_{\epsilon}}{r}+\partial_r^2\alpha_{\epsilon}-\frac{\alpha_{\epsilon}}{r^2}\right)\sin\theta\,.
\end{aligned}
\end{equation*}
If we choose 
\begin{equation*}
p_{\epsilon}(r)=\int_{\rho}^r\frac{\alpha_{\epsilon}(s)^2}{s}ds
\end{equation*}
and divide by $\sin\theta$, we end up with the parabolic equation
\begin{equation}\label{parabolic}
\partial_t\alpha_{\epsilon}=\epsilon\left(\frac{\partial_r\alpha_{\epsilon}}{r}+\partial_r^2\alpha_{\epsilon}-\frac{\alpha_{\epsilon}}{r^2}\right).
\end{equation}
Insertion of our ansatz into the second equation of (\ref{navierstokes}) also gives (\ref{parabolic}), as one can easily check by a similar computation. Moreover, the divergence-free condition is automatically satisfied, the initial condition becomes
\begin{equation}\label{initialparabolic}
\alpha_{\epsilon}(\cdot,0)=\alpha_0
\end{equation}
with $\alpha_0$ defined by (\ref{initial}), and the boundary condition translates into 
\begin{equation}\label{boundaryparabolic}
\alpha_{\epsilon}(\rho)=\alpha_{\epsilon}(R)=0.
\end{equation}
Thus we obtain the well-posed parabolic initial and boundary value problem (\ref{parabolic}), (\ref{initialparabolic}), (\ref{boundaryparabolic}). By well-known results (cf. e.g. \cite{evans}, Section 7.1), this parabolic problem admits, for each $\epsilon>0$, a unique weak solution. But our calculations so far show that, if $\alpha_{\epsilon}$ is a solution to the parabolic problem, then the corresponding $v_{\epsilon}$ defined by (\ref{ansatz}) is the (unique) Leray-Hopf solution of the Navier-Stokes problem (\ref{navierstokes}), and at the same time it satisfies the initial and boundary value problem for the heat equation:
\begin{equation*}
\begin{aligned}
\partial_tv_{\epsilon}&=\epsilon\Delta v_{\epsilon}\\
\operatorname{div}v_{\epsilon}&=0\\
v_{\epsilon}(\cdot,0)&=v_0\\
v_{\epsilon}|_{\partial\Omega}&=0.
\end{aligned}
\end{equation*} 
Since the solutions of the heat equation converge strongly to the stationary solution, and since we have shown that for our particular initial data the heat equation coincides with the Navier-Stokes equations, the proposition is thus proved. 
\end{proof}
\begin{rem}
The previous discussion can be extended to initial data on a cylinder of the form $Z=\Omega\times\mathbb{T}\subset\Real^2\times\mathbb{T}$, where $\Omega\subset\Real^2$ is still the annulus. Indeed, for so-called 2 1/2 dimensional initial data $V_0(x_1,x_2)=(v_0(x_1,x_2),w(x_1,x_2))$ on $Z$, where $v_0$ is as in (\ref{rotational}), there may exist infinitely many admissible weak solutions, but only the solution given by
\begin{equation*}
V(x_1,x_2,t)=(v_0(x_1,x_2),w(x_1-(v_0)_1t,x_2-(v_0)_2t))
\end{equation*}
arises as a viscosity limit. We omit details, but remark that this can be shown along the lines of \cite{shearflow}, where a similar analysis was carried out for the case of shear flows.
\end{rem}

\section{Dissipative Solutions}\label{dissipative}
Let $S(w)=\frac{1}{2}(\nabla w+\nabla w^t)$ denote the symmetric gradient of a vectorfield $w$, and set 
\begin{equation*}
E(w)=-\partial_tw-P(w\cdot\nabla w),
\end{equation*}
with $P$ denoting the Leray-Helmholtz projection onto $H(\Omega)$. 
  
The following definition is from \cite{lions}, given here in the version of \cite{bardostiti} for bounded domains. The reader may consult these references also for a motivation of the definition.
\begin{defn}
Let $\Omega$ be a bounded domain with $C^1$ boundary. A vectorfield $v\in C([0,T];H_w(\Omega))$ is said to be a \emph{dissipative solution} of the Euler equations (\ref{euler}) if for every divergence-free test vectorfield $w\in C^1(\overline{\Omega}\times[0,T])$ with $w\cdot\nu\restriction_{\partial\Omega}=0$ one has
\begin{equation}\label{defdissipative}
\begin{aligned}
\int_{\Omega}|v-w|^2dx&\leq\operatorname{exp}\left(2\int_0^t\norm{S(w)}_{\infty}ds\right)\int_{\Omega}|v(x,0)-w(x,0)|^2dx\\
+&2\int_0^t\int_{\Omega}\operatorname{exp}\left(2\int_s^t\norm{S(w)}_{\infty}d\tau\right)E(w)\cdot(v-w)dxds
\end{aligned}
\end{equation}
for all $t\in[0,T]$.
\end{defn}
An immediate consequence of this definition is the weak-strong uniqueness (Proposition 4.1 in \cite{lions}):
\begin{prop}\label{weak-strong}
Suppose there exists a solution $v\in C^1(\overline{\Omega}\times[0,T])$ of the Euler equations (\ref{euler}). Then $v$ is unique in the class of dissipative solutions with the same initial data.
\end{prop}
This follows simply by choosing $w=v$ as a test function in the definition of dissipative solutions.

Next, we prove Corollary \ref{cor}, showing that admissible solutions may fail to be unique in bounded domains even for smooth initial data.  

\begin{proof}
Recall the construction from Section \ref{nonuniqueness} and define
\begin{equation*}
\tilde{\Omega}=\{x\in\Real^2: \rho<|x|<r_0\}\subset \Omega.
\end{equation*}
It follows immediately from the definition that the restriction of a subsolution to a subdomain is itself a subsolution. Therefore we may consider the subsolution $(\bar{v},\bar{u},\bar{q})$ constructed in Section \ref{nonuniqueness} as a subsolution on $\tilde\Omega$ with energy density $\bar{e}$ as in Proposition \ref{p:subsol}, with initial data given by
$$
\bar{v}(x,0)=-\frac{x^\perp}{|x|^3}\quad\textrm{ for }x\in\tilde\Omega
$$
(c.f. \eqref{e:v0}). Applying this time Theorem \ref{t:criterion} in $\tilde\Omega$ with this subsolution yields infinitely many admissible weak solutions as in the proof of Theorem \ref{wildrotation}.

Since the initial data $\bar{v}(x,0)$ is smooth on $\tilde\Omega$, there exists a unique strong solution (indeed, this is the stationary solution). Thus weak-strong uniqueness fails, a fortiori implying that the non-stationary weak admissible solutions are not dissipative in the sense of Lions.
\end{proof}

\section{A Criterion for Admissible Solutions to be Dissipative}\label{holder}
We have seen that, on bounded domains, an admissible weak solution may fail to be dissipative. However this will not happen provided such a solution is H\"{o}lder continuous near the boundary of the domain, as claimed in Theorem \ref{criterion} above. The aim of this last section is to prove this theorem. We follow Appendix B of \cite{euler2}, but have to take into account that we need to deal with test functions which are not necessarily compactly supported in $\Omega$ in the definition of dissipative solutions.

So let $\Omega$ be a bounded domain in $\Real^2$ with $C^2$ boundary and $v$ an admissible weak solution of the Euler equations (\ref{euler}) as in the statement of Theorem \ref{criterion}. Assume for the moment that for every divergence-free $w\in C^1(\overline{\Omega}\times[0,T])$ satisfying the boundary condition we have
\begin{equation}\label{identity}
\frac{d}{dt}\int_{\Omega}v\cdot wdx=\int_{\Omega}\left(S(w)(v-w)\cdot(v-w)-E(w)\cdot v\right)dx
\end{equation}    
in the sense of distributions, where $E(w)$ is the quantity defined at the beginning of Section \ref{dissipative}. We claim that (\ref{identity}) implies already that $v$ is a dissipative solution. Indeed this can be shown exactly as in \cite{euler2}: On the one hand, since $v$ is admissible, 
\begin{equation}\label{identity2}
\frac{d}{dt}\int_{\Omega}|v|^2dx\leq0
\end{equation}
in the sense of distributions. On the other hand, using the definition of $E(w)$ and the identity $\int_{\Omega}(w\cdot\nabla w)\cdot wdx=0$ (which follows from $w\cdot\nu\restriction_{\partial\Omega}=0$), we have
\begin{equation}\label{identity3}
\frac{d}{dt}\int_{\Omega}|w|^2dx=-2\int_{\Omega}E(w)\cdot wdx.
\end{equation}
Since
\begin{equation*}
\int_{\Omega}|v-w|^2dx=\int_{\Omega}|v|^2dx+\int_{\Omega}|w|^2dx-2\int_{\Omega}v\cdot wdx,
\end{equation*}
we infer from this together with (\ref{identity}), (\ref{identity2}), and (\ref{identity3}) that
\begin{equation*}
\begin{aligned}
\frac{d}{dt}\int_{\Omega}|v-w|^2dx&\leq2\int_{\Omega}\left(E(w)\cdot(v-w)-S(w)(v-w)\cdot(v-w)\right)dx\\
&\leq2\int_{\Omega}E(w)\cdot(v-w)dx+2\norm{S(w)}_{\infty}\int_{\Omega}|v-w|^2dx
\end{aligned}
\end{equation*}
in the sense of distributions. We can then apply Gr\"{o}nwall's inequality as in \cite{euler2} to obtain (\ref{defdissipative}) for every $t\in[0,T]$. Therefore, it remains to prove (\ref{identity}) for every test function $w$.

In \cite{euler2}, identity (\ref{identity}) is proved for the case that $w$ is compactly supported in $\Omega$ at almost every time (see the considerations after equality (96) in \cite{euler2}). Let now $w\in C^1(\overline{\Omega}\times[0,T])$ be a divergence-free vectorfield with $w\cdot\nu|_{\partial\Omega}=0$, which does not necessarily have compact support in space. We will suitably approximate $w$ by vectorfields that do have compact support, much in the spirit of T. Kato \cite{kato} (in particular Section 4 therein).     

Assume for the moment that $\Omega$ is simply connected, so that $\partial\Omega$ has only one connected component. Since $w$ is divergence-free, there exists a function $\psi\in C([0,T];C^2(\overline{\Omega}))\cap C^1(\overline{\Omega}\times[0,T])$ such that 
\begin{equation*}
w(x,t)=\nabla^{\perp}\psi(x,t)
\end{equation*}
and $\psi\restriction_{\partial\Omega}=0$.
Let now $\chi:[0,\infty)\rightarrow\Real$ be a nonnegative smooth function such that
\begin{equation*}
\chi(s)=\begin{cases}0 & \text{if $s<1$}\\
1 & \text{if $s>2$}
\end{cases}
\end{equation*}
and set 
\begin{equation*}
w_{\epsilon}(x,t)=\nabla^{\perp}\left(\chi\left(\frac{\operatorname{dist}(x,\partial\Omega)}{\epsilon}\right)\psi(x,t)\right).
\end{equation*}
Then, by Lemma 14.16 in \cite{gilbargtrudinger}, there exists $\eta>0$ depending on $\Omega$ such that $x\mapsto\operatorname{dist}(x,\partial\Omega)$ is $C^2$ on 
\begin{equation}\label{boundarystrip}
\Gamma_{\eta}=\{x\in\overline{\Omega}:\operatorname{dist}(x,\partial\Omega)<\eta\},
\end{equation}
and hence $w_{\epsilon}\in C^1_c(\Omega\times[0,T])$ for sufficiently small $\epsilon>0$. Therefore, (\ref{identity}) is true for $w_{\epsilon}$:
\begin{equation}\label{identityeps}
\frac{d}{dt}\int_{\Omega}v\cdot w_{\epsilon}dx=\int_{\Omega}\left(S(w_{\epsilon})(v-w_{\epsilon})\cdot(v-w_{\epsilon})-E(w_{\epsilon})\cdot v\right)dx.
\end{equation}   
We will now let $\epsilon$ tend to zero in order to recover (\ref{identity}). 

Writing $d(x)=\operatorname{dist}(x,\partial\Omega)$, we have from the definition of $w_{\epsilon}$:
\begin{equation}\label{productrule}
w_{\epsilon}=\chi\left(\frac{d}{\epsilon}\right)\nabla^{\perp}\psi+\frac{1}{\epsilon}\chi'\left(\frac{d}{\epsilon}\right)\psi\nabla^{\perp}d,
\end{equation}
and since $\psi\in C([0,T];C^2(\overline{\Omega}))$ and $\psi\restriction_{\partial\Omega}=0$, there is a constant $C$ independent of $t$ and $\epsilon$ such that
\begin{equation*}
|\psi(x,t)|\leq Cd(x)
\end{equation*}
for all $x\in\overline{\Omega}$. Moreover, as the support of $\chi'\left(\frac{\cdot}{\epsilon}\right)$ is contained in $(\epsilon,2\epsilon)$, and as $|\nabla d|\leq1$, it follows from (\ref{productrule}) that
\begin{equation}\label{strong}
w_{\epsilon}\to w\hspace{0.3cm}\text{strongly in $L^{\infty}([0,T];L^2(\Omega))$}
\end{equation}
as $\epsilon\to0$. For the left hand side of (\ref{identityeps}) this immediately implies
\begin{equation*}
\frac{d}{dt}\int_{\Omega}v\cdot w_{\epsilon}dx\to\frac{d}{dt}\int_{\Omega}v\cdot wdx
\end{equation*}
in the sense of distributions. Moreover, the right hand side of (\ref{identityeps}) can be written, recalling the definition of $E(w_{\epsilon})$, as
\begin{equation*}
\begin{aligned}
\int_{\Omega}&\left(S(w_{\epsilon})(v-w_{\epsilon})\cdot(v-w_{\epsilon})-E(w_{\epsilon})\cdot v\right)dx\\
&=\int_{\Omega}\left[\partial_tw_{\epsilon}\cdot v+(v\cdot\nabla w_{\epsilon})\cdot v-((v-w_{\epsilon})\cdot\nabla w_{\epsilon})\cdot w_{\epsilon}\right]dx,
\end{aligned}
\end{equation*}
and the right hand side of (\ref{identity}) is given by a similar expression.
 
Next, observe that, again by (\ref{strong}),
\begin{equation*}
\int_{\Omega}\partial_tw_{\epsilon}\cdot vdx\to\int_{\Omega}\partial_tw\cdot vdx
\end{equation*}
in the sense of distributions and also that
\begin{equation*}
\int_{\Omega}((v-w_{\epsilon})\cdot\nabla w_{\epsilon})\cdot w_{\epsilon}dx=\int_{\Omega}((v-w)\cdot\nabla w)\cdot wdx=0
\end{equation*}
thanks to the formula $(v-w)\cdot\nabla w)\cdot w=(v-w)\cdot\frac{1}{2}\nabla|w|^2$ and the fact that $v-w\in H(\Omega)$ (and similarly for $((v-w_{\epsilon})\cdot\nabla w_{\epsilon})\cdot w_{\epsilon}$).

To complete the proof of (\ref{identity}) and therefore of Theorem \ref{criterion}, it remains to show that
\begin{equation}\label{problemterm}
\int_{\Omega}(v\cdot\nabla w_{\epsilon})\cdot vdx\to\int_{\Omega}(v\cdot\nabla w)\cdot vdx
\end{equation}
in the sense of distributions as $\epsilon\to0$. 

To this end, note that for every $x\in\Omega$ sufficiently close to $\partial\Omega$ there exists a unique closest point $\hat{x}\in\partial\Omega$, and then 
\begin{equation*}
x=\hat{x}+d(x)\nu(\hat{x}).
\end{equation*} 
We denote by $\tau(\hat{x})=\left(-\nu_2(\hat{x}),\nu_1(\hat{x})\right)$ the unit vector at $\hat{x}$ tangent to $\partial\Omega$ and use the notation
$v_{\tau}(x)=v(x)\cdot\tau(\hat{x})$, $\partial_{\tau}w_{\nu}(x)=\nabla w_{\nu}(x)\cdot\tau(\hat{x})$, etc. (recall that $\hat{x}$ is uniquely determined by $x$). If $\epsilon$ is sufficiently small, we can then write (recall (\ref{boundarystrip}))
\begin{equation*}
\begin{aligned}
\int_{\Omega}(v\cdot\nabla(w_{\epsilon}-w))\cdot vdx&=\int_{\Gamma_{2\epsilon}}v_{\nu}\partial_{\nu}(w_{\epsilon}-w)_{\nu}v_{\nu}dx+\int_{\Gamma_{2\epsilon}}v_{\nu}\partial_{\nu}(w_{\epsilon}-w)_{\tau}v_{\tau}dx\\
+\int_{\Gamma_{2\epsilon}}v_{\tau}&\partial_{\tau}(w_{\epsilon}-w)_{\nu}v_{\nu}dx+\int_{\Gamma_{2\epsilon}}v_{\tau}\partial_{\tau}(w_{\epsilon}-w)_{\tau}v_{\tau}dx\\
&=:I_1+I_2+I_3+I_4.
\end{aligned}
\end{equation*}
Recalling (\ref{productrule}) as well as $\nabla^{\perp}\psi=w$ and observing $\nabla d=\nu$, we compute
\begin{equation*}
\begin{aligned}
(w_{\epsilon}-w)_{\nu}&=\left(\chi\left(\frac{d}{\epsilon}\right)-1\right)\partial_{\tau}\psi\\
&=\left(\chi\left(\frac{d}{\epsilon}\right)-1\right)w_{\nu},
\end{aligned}
\end{equation*}
\begin{equation*}
\begin{aligned}
(w_{\epsilon}-w)_{\tau}&=-\left(\chi\left(\frac{d}{\epsilon}\right)-1\right)\partial_{\nu}\psi+\frac{1}{\epsilon}\chi'\left(\frac{d}{\epsilon}\right)\psi\\
&=\left(\chi\left(\frac{d}{\epsilon}\right)-1\right)w_{\tau}+\frac{1}{\epsilon}\chi'\left(\frac{d}{\epsilon}\right)\psi,
\end{aligned}
\end{equation*}
\begin{equation}\label{nunu}
\partial_{\nu}(w_{\epsilon}-w)_{\nu}=\frac{1}{\epsilon}\chi'\left(\frac{d}{\epsilon}\right)w_{\nu}+\left(\chi\left(\frac{d}{\epsilon}\right)-1\right)\partial_{\nu}w_{\nu},
\end{equation}
\begin{equation}\label{nutau}
\partial_{\nu}(w_{\epsilon}-w)_{\tau}=\left(\chi\left(\frac{d}{\epsilon}\right)-1\right)\partial_{\nu}w_{\tau}+\frac{1}{\epsilon^2}\chi''\left(\frac{d}{\epsilon}\right)\psi,
\end{equation}
\begin{equation}\label{taunu}
\partial_{\tau}(w_{\epsilon}-w)_{\nu}=\left(\chi\left(\frac{d}{\epsilon}\right)-1\right)\partial_{\tau}w_{\nu},
\end{equation}
\begin{equation}\label{tautau}
\partial_{\tau}(w_{\epsilon}-w)_{\tau}=\left(\chi\left(\frac{d}{\epsilon}\right)-1\right)\partial_{\tau}w_{\tau}+\frac{1}{\epsilon}\chi'\left(\frac{1}{\epsilon}\right)w_{\nu}.
\end{equation}
Before we estimate $I_1$-$I_4$ using (\ref{nunu})-(\ref{tautau}), let us collect some more information: As mentioned above, there is a constant $C$ independent of $t$ such that $|\psi(x)|\leq Cd(x)$. Moreover, since $w\in C^1(\overline{\Omega}\times[0,T])$ and $w_{\nu}=0$ on $\partial\Omega$, we find similarly a constant independent of $t$ such that $|w_{\nu}(x)|\leq Cd(x)$. By assumption, if $\epsilon$ is small enough, then $v$ is H\"{o}lder continuous with exponent $\alpha$ on $\Gamma_{2\epsilon}$, uniformly in $t$, and since $v\in H(\Omega)$ implies $v_{\nu}=0$ on $\partial\Omega$ (cf. \cite{galdibook} Chapter III), we obtain another time-independent constant such that $|v_{\nu}(x)|\leq Cd(x)^{\alpha}$ on $\Gamma_{2\epsilon}$. Finally note that $\psi$, $w$, and $v$ are uniformly bounded on $\Gamma_{2\epsilon}$ provided $\epsilon$ is small, and that there is a constant independent of $\epsilon$ such that $|\Gamma_{2\epsilon}|\leq C\epsilon$.

In the light of these considerations we can use (\ref{nunu})-(\ref{tautau}) to estimate
\begin{equation*}
\begin{aligned}
|I_1|&\leq\frac{1}{\epsilon}\int_{\Gamma_{2\epsilon}}v_{\nu}^2\norm{\chi'}_{\infty}\norm{w_{\tau}}_{\infty}dx+\int_{\Gamma_{2\epsilon}}v_{\nu}^2\norm{\chi-1}_{\infty}\norm{\partial_{\nu}w_{\nu}}_{\infty}dx\\
&\leq C\epsilon^{2\alpha+1}+C\epsilon^{2\alpha+1}\to0,
\end{aligned}
\end{equation*}   
\begin{equation*}
\begin{aligned}
|I_2|&\leq\int_{\Gamma_{2\epsilon}}|v_{\nu}||v_{\tau}|\norm{\chi-1}_{\infty}\norm{\partial_{\nu}w_{\tau}}_{\infty}dx+\frac{1}{\epsilon^2}\int_{\Gamma_{2\epsilon}}|v_{\nu}||v_{\tau}|\norm{\chi''}_{\infty}|\psi|dx\\
&\leq C\epsilon^{\alpha+1}+C\epsilon^{\alpha}\to0,
\end{aligned}
\end{equation*}
\begin{equation*}
|I_3|\leq\int_{\Gamma_{2\epsilon}}|v_{\nu}||v_{\tau}|\norm{\chi-1}_{\infty}\norm{\partial_{\tau}w_{\nu}}_{\infty}dx\leq C\epsilon^{\alpha+1}\to0,
\end{equation*}
\begin{equation*}
\begin{aligned}
|I_4|&\leq\int_{\Gamma_{2\epsilon}}v_{\tau}^2\norm{\chi-1}_{\infty}\norm{\partial_{\tau}w_{\tau}}_{\infty}dx+\frac{1}{\epsilon}\int_{\Gamma_{2\epsilon}}v_{\tau}^2\norm{\chi'}_{\infty}|w_{\nu}|\\
&\leq C\epsilon+C\epsilon\to0,
\end{aligned}
\end{equation*}
all estimates being uniform in time. This proves Theorem \ref{criterion} if $\Omega$ is simply connected.

As a final step, we convince ourselves that the proof can easily be modified to the general case when $\partial\Omega$ has $N$ connected components $\Gamma^1,\ldots,\Gamma^N$ in the spirit of Section 1.4 of \cite{kato2}. There still exists $\psi\in C([0,T];C^2(\overline{\Omega}))\cap C^1(\overline{\Omega}\times[0,T])$ with $\nabla^{\perp}\psi=w$, but we can no longer require $\psi\restriction_{\partial\Omega}=0$. Instead, $\psi$ will take the constant value $\psi^i$ on $\Gamma^i$, but the numbers $\psi^i$ may be different. Now, if $\epsilon>0$ is small enough, then the sets
\begin{equation*}
\Gamma^i_{2\epsilon}=\{x\in\overline\Omega:\operatorname{dist}(x,\Gamma^i)<2\epsilon\},\hspace{0.3cm}i=1,\ldots,N,
\end{equation*}
will be mutually disjoint, so that $w_{\epsilon}$ is well-defined by setting
\begin{equation*}
w_{\epsilon}(x)=\begin{cases}\nabla^{\perp}\left(\chi\left(\frac{\operatorname{dist}(x,\partial\Omega)}{\epsilon}\right)(\psi(x)-\psi^i)\right) & \text{if $x\in\Gamma^i_{2\epsilon}$}\\
w(x) & \text{if $x\in\overline\Omega\setminus\bigcup_i\Gamma^i_{2\epsilon}$}
\end{cases}
\end{equation*}
with $\chi$ as in the simply connected case. With this choice of $w_{\epsilon}$ we can then employ the very same arguments as above.\qed

\begin{rem}
Theorem \ref{criterion} implies that there can not be wild solutions on an annulus with smooth rotational initial data that are H\"{o}lder continuous. Indeed, any admissible H\"{o}lder continuous solution must be dissipative by our theorem, and the weak-strong uniqueness then yields that this solution must coincide with the stationary one. This observation is particularly interesting in the light of recent results (e.g. \cite{isett, bdlsz, daneri}) where examples of H\"{o}lder continuous wild solutions are constructed.  
\end{rem}
\vskip0.1cm
{\it One of the first papers of Professor Mark Vishik ``On general boundary problems for elliptic differential equations''  \cite{Visik} was essential,  in particular in France, for the training of mathematicians in the generation of the first author of this contribution. Then when he turned to Navier-Stokes and turbulence he took an important role in progress over the last 60 years toward the mathematical understanding of turbulence in fluid  mechanics. Hence we hope that this essay will contribute to his memory and to the recognition of his influence on our community.}
\vskip0.1cm
\textbf{Acknowledgements.} The authors would like to thank Professor Edriss Titi for interesting and valuable discussions. 

The research of L. Sz. is supported by ERC Grant Agreement No. 277993. 
Part of this work was done while E. W. was a visitor to the project ``Instabilities in Hydrodynamics'' of the Fondation Sciences Math\'{e}matiques de Paris. He gratefully acknowledges the Fondation's support. 

\bibliography{Rotation}

\begin{thebibliography}{10}

\bibitem{amick}
Charles~J. Amick.
\newblock Existence of solutions to the nonhomogeneous steady {N}avier-{S}tokes
  equations.
\newblock {\em Indiana Univ. Math. J.}, 33(6):817--830, 1984.

\bibitem{bardostiti}
C.~Bardos and E.~S. Titi.
\newblock Euler equations for an ideal incompressible fluid.
\newblock {\em Uspekhi Mat. Nauk}, 62(3(375)):5--46, 2007.

\bibitem{bardostiti2}
C.~Bardos and E.~S. Titi.
\newblock Mathematics and turbulence: Where do we stand?
\newblock {\em {P}reprint}, 2013.

\bibitem{shearflow}
Claude Bardos, Edriss~S. Titi, and Emil Wiedemann.
\newblock The vanishing viscosity as a selection principle for the {E}uler
  equations: the case of 3{D} shear flow.
\newblock {\em C. R. Math. Acad. Sci. Paris}, 350(15-16):757--760, 2012.

\bibitem{brenierdelellissz}
Yann Brenier, Camillo De~Lellis, and L{\'a}szl{\'o} Sz{\'e}kelyhidi, Jr.
\newblock Weak-strong uniqueness for measure-valued solutions.
\newblock {\em Comm. Math. Phys.}, 305(2):351--361, 2011.

\bibitem{bdlsz}
T.~Buckmaster, C.~De~Lellis, and L.~Sz\'{e}kelyhidi.
\newblock Transporting microstructure and dissipative {E}uler flows.
\newblock {\em {P}reprint}, 2013.

\bibitem{isentropic}
E.~Chiodaroli, C.~De~Lellis, and O.~Kreml.
\newblock Global ill-posedness of the isentropic system of gas dynamics.
\newblock {\em {P}reprint}, 2013.

\bibitem{DafermosBook}
Constantine~M. Dafermos.
\newblock {\em Hyperbolic conservation laws in continuum physics}, volume 325
  of {\em Grundlehren der Mathematischen Wissenschaften [Fundamental Principles
  of Mathematical Sciences]}.
\newblock Springer-Verlag, Berlin, third edition, 2010.

\bibitem{daneri}
Sara Daneri.
\newblock Cauchy problem for dissipative {H}\"{o}lder solutions to the
  incompressible {E}uler equations.
\newblock {\em {P}reprint}, 2013.

\bibitem{euler1}
Camillo De~Lellis and L{\'a}szl{\'o} Sz{\'e}kelyhidi, Jr.
\newblock The {E}uler equations as a differential inclusion.
\newblock {\em Ann. of Math. (2)}, 170(3):1417--1436, 2009.

\bibitem{euler2}
Camillo De~Lellis and L{\'a}szl{\'o} Sz{\'e}kelyhidi, Jr.
\newblock On admissibility criteria for weak solutions of the {E}uler
  equations.
\newblock {\em Arch. Ration. Mech. Anal.}, 195(1):225--260, 2010.

\bibitem{hprinciple}
Camillo De~Lellis and L{\'a}szl{\'o} Sz{\'e}kelyhidi, Jr.
\newblock The {$h$}-principle and the equations of fluid dynamics.
\newblock {\em Bull. Amer. Math. Soc. (N.S.)}, 49(3):347--375, 2012.

\bibitem{DuchonRobert}
Jean Duchon and Raoul Robert.
\newblock Inertial energy dissipation for weak solutions of incompressible
  {E}uler and {N}avier-{S}tokes equations.
\newblock {\em Nonlinearity}, 13(1):249--255, 2000.

\bibitem{evans}
Lawrence~C. Evans.
\newblock {\em Partial differential equations}, volume~19 of {\em Graduate
  Studies in Mathematics}.
\newblock American Mathematical Society, Providence, RI, second edition, 2010.

\bibitem{EyinkSreenivasan}
Gregory~L. Eyink and Katepalli~R. Sreenivasan.
\newblock Onsager and the theory of hydrodynamic turbulence.
\newblock {\em Rev. Modern Phys.}, 78(1):87--135, 2006.

\bibitem{galdibook}
Giovanni~P. Galdi.
\newblock {\em An introduction to the mathematical theory of the
  {N}avier-{S}tokes equations. {V}ol. {I}}, volume~38 of {\em Springer Tracts
  in Natural Philosophy}.
\newblock Springer-Verlag, New York, 1994.
\newblock Linearized steady problems.

\bibitem{galdi}
Giovanni~P. Galdi.
\newblock An introduction to the {N}avier-{S}tokes initial-boundary value
  problem.
\newblock In {\em Fundamental directions in mathematical fluid mechanics}, Adv.
  Math. Fluid Mech., pages 1--70. Birkh\"auser, Basel, 2000.

\bibitem{gilbargtrudinger}
David Gilbarg and Neil~S. Trudinger.
\newblock {\em Elliptic partial differential equations of second order}.
\newblock Classics in Mathematics. Springer-Verlag, Berlin, 2001.
\newblock Reprint of the 1998 edition.

\bibitem{isett}
Philip Isett.
\newblock H\"{o}lder continuous {E}uler flows in three dimensions with compact
  support in time.
\newblock {\em {P}reprint}, 2013.

\bibitem{kato2}
Tosio Kato.
\newblock On classical solutions of the two-dimensional non-stationary euler
  equation.
\newblock {\em Archive for Rational Mechanics and Analysis}, 25(3):188--200,
  1967.

\bibitem{kato}
Tosio Kato.
\newblock Remarks on zero viscosity limit for nonstationary navier-stokes flows
  with boundary.
\newblock In {\em Seminar on nonlinear partial differential equations
  (Berkeley, Calif., 1983)}, volume~2, pages 85--98. Springer New York, 1984.

\bibitem{lions}
Pierre-Louis Lions.
\newblock {\em Mathematical topics in fluid mechanics. {V}ol. 1}, volume~3 of
  {\em Oxford Lecture Series in Mathematics and its Applications}.
\newblock The Clarendon Press Oxford University Press, New York, 1996.
\newblock Incompressible models, Oxford Science Publications.

\bibitem{bertozzimajda}
Andrew~J. Majda and Andrea~L. Bertozzi.
\newblock {\em Vorticity and incompressible flow}, volume~27 of {\em Cambridge
  Texts in Applied Mathematics}.
\newblock Cambridge University Press, Cambridge, 2002.

\bibitem{saintraymond}
Laure Saint-Raymond.
\newblock Convergence of solutions to the {B}oltzmann equation in the
  incompressible {E}uler limit.
\newblock {\em Arch. Ration. Mech. Anal.}, 166(1):47--80, 2003.

\bibitem{scheffer}
Vladimir Scheffer.
\newblock An inviscid flow with compact support in space-time.
\newblock {\em J. Geom. Anal.}, 3(4):343--401, 1993.

\bibitem{shnirel1}
A.~Shnirelman.
\newblock On the nonuniqueness of weak solution of the {E}uler equation.
\newblock {\em Comm. Pure Appl. Math.}, 50(12):1261--1286, 1997.

\bibitem{shnirel2}
A.~Shnirelman.
\newblock Weak solutions with decreasing energy of incompressible {E}uler
  equations.
\newblock {\em Comm. Math. Phys.}, 210(3):541--603, 2000.

\bibitem{shnirel3}
Alexander~I. Shnirelman.
\newblock Lattice theory and flows of ideal incompressible fluid.
\newblock {\em Russian J. Math. Phys.}, 1(1):105--114, 1993.

\bibitem{szlecturenotes}
L.~Sz\'{e}kelyhidi.
\newblock From isometric embeddings to turbulence.
\newblock {\em Lecture Notes}, 2012.

\bibitem{vortexpaper}
L{\'a}szl{\'o} Sz{\'e}kelyhidi.
\newblock Weak solutions to the incompressible {E}uler equations with vortex
  sheet initial data.
\newblock {\em C. R. Math. Acad. Sci. Paris}, 349(19-20):1063--1066, 2011.

\bibitem{euleryoung}
L{\'a}szl{\'o} Sz{\'e}kelyhidi and Emil Wiedemann.
\newblock Young measures generated by ideal incompressible fluid flows.
\newblock {\em Arch. Ration. Mech. Anal.}, 206(1):333--366, 2012.

\bibitem{Visik}
M.~I. Vishik.
\newblock On general boundary problems for elliptic differential equations.
\newblock {\em Trudy Moskov. Mat. Ob\v s\v c.}, 1:187--246, 1952.

\bibitem{eulerexistence}
Emil Wiedemann.
\newblock Existence of weak solutions for the incompressible {E}uler equations.
\newblock {\em Ann. Inst. H. Poincar\'e Anal. Non Lin\'eaire}, 28(5):727--730,
  2011.

\end{thebibliography}
\end{document}